\documentclass[11pt,reqno]{amsart}

\usepackage[margin = 1.3in]{geometry}
\usepackage[T1]{fontenc}
\usepackage{lmodern}
\usepackage{microtype}
\usepackage{amsmath,amssymb,mathtools}
\usepackage{enumitem}
\usepackage{booktabs}
\usepackage{hyperref}

\newtheorem{theorem}{Theorem}[section]
\newtheorem{proposition}[theorem]{Proposition}
\newtheorem{lemma}[theorem]{Lemma}
\newtheorem{corollary}[theorem]{Corollary}
\newtheorem{problem}[theorem]{Problem}

\theoremstyle{definition}

\theoremstyle{remark}

\DeclareMathOperator{\Aut}{Aut}

\DeclareMathOperator{\Soc}{soc}

\DeclareMathOperator{\PSp}{PSp}

\newcommand{\normal}{\trianglelefteq}

\newcommand{\cP}{\mathcal P}

\title[Products of nonconjugate maximal subgroups]{Products of Nonconjugate Maximal Subgroups and Solvability}

\author[J. B. Li]{Jinbao Li} \address{School  of Mathematics and Physics, Suqian University, Suqian, Jiangsu 223800, China.}
\email{leejinbao25@hotmail.com}

\author[Y. Yang]{Yong Yang}
\address{Department of Mathematics, Texas State University, 601 University Drive, San Marcos, TX 78666, USA.}
\email{yang@txstate.edu}

\date{}

\begin{document}

\begin{abstract}
We prove that a finite group $G$ is solvable whenever $MN=G$ for
every pair of nonconjugate maximal subgroups $M,N<G$. Equivalently,
every finite nonsolvable group has two nonconjugate maximal
subgroups whose setwise product is proper. This gives a negative
answer to Problem 10.34 of the Kourovka Notebook. As an application, we also answer an open question raised by Guo.\\
{\bf Keywords:} maximal subgroup, group factorization, solvable
group, almost simple group, primitive group \\
{\bf AMS Mathematics Subject Classification(2010):} 20D05, 20D10,
20D60.
\end{abstract}

\maketitle

\section{Introduction}

All groups considered are finite groups.

Let $G$ be a solvable group. Then any two maximal subgroups of $G$
are conjugate in $G$ if and only if they have the same core in $G$
(cf. \cite[2.4.3]{Guo2}). It follows  that $G$ coincides with the
product of any two nonconjugate maximal subgroups of $G$. An
interesting question is whether the converse holds. In 1986, along this line, the following problem was proposed
by V.S. Monakhov in the 10th Issue of  Kourovka Notebook  (see \cite[Problem 10.34]{Kourovka}).

\begin{problem}
  Does there exist a non-solvable  finite group which coincides with the product
of any two of its nonconjugate maximal subgroups?
\end{problem}

It is known from editors' comment following \cite[Problem 10.34]{Kourovka} that this problem was partially solved by T.V. Tikhonenko
and V.N. Tyutyanov in \cite{TT} for almost simple groups. They  proved that no almost simple group has this property.  Our main result is the following.

\begin{theorem}\label{thm:main}
Let $G$ be a finite group. Suppose that
\begin{equation*}\label{eq:property}
MN=G
\end{equation*}
whenever $M$ and $N$ are nonconjugate maximal subgroups of $G$. Then
$G$ is solvable.
\end{theorem}

By Theorem~\ref{thm:main}, we obtain the following corollary.

\begin{corollary}\label{cor:kourovka}
There is no finite nonsolvable group which is the product of every
two of its nonconjugate maximal subgroups. Thus Problem 10.34 of the
Kourovka Notebook has a negative answer.
\end{corollary}

There is a closely related problem of W.B. Guo on the generalized permutability of subgroups and the product of maximal subgroups in \cite{Guo2},
which can be formulated as follows. Let $A$ and $B$ be subgroups
of a group $G$. Then $A$ is said to be permutable with $B$ if
$AB=BA$. It may happen that $AB\neq BA$, but $AB^x=B^xA$ for some
element $x\in G$. For example, two conjugate maximal subgroups $M$
and $L$ of a group $G$ cannot permute in general, but we always
have that $ML^g=L^gM=M$ for some element $g\in G$. Based on this
observation, the following concepts about generalized permutable subgroups were introduced.  Let $X$ be a non-empty subset of a group $G$ and $A$ and $B$
be two subgroups of $G$. If $AB^x=B^xA$ for some $x \in X$, then
$A$ is said to be $X$-permutable with $B$; and if $X=G$, then $A$ is
also said to be conditionally permutable (or in brevity,
$c$-permutable) with $B$ in $G$ (see \cite[ Definition
2.1.1]{Guo2}). Clearly, any two conjugate subgroups are
$c$-permutable. Moreover, there are lots of $c$-permutable subgroups in nonabelian simple groups. For example, any element of order $2$ in the alternating group $A_5$ is $c$-permutable with all subgroups of $A_5$.  The influence of $c$-permutable subgroups on the structure of solvable groups is widely studied (see \cite[Chapter 2]{Guo2}). For example, a group $G$ is solvable if and only if every two Hall subgroups of $G$ are $c$-permutable in $G$ (see \cite[Lemma 5.1]{Guo2}). In this vein, the following problem was proposed by  Guo (see \cite[Problem 2.6.10]{Guo2}).

\begin{problem} \label{P'}
Is a group $G$ solvable if any two maximal subgroups of $G$ are
$c$-permutable in $G$?
\end{problem}

By our Theorem \ref{thm:main}, we can prove this is true.

\begin{theorem} \label{thm:main2}
A group $G$ is solvable if any two maximal subgroups are $c$-permutable in $G$. 
\end{theorem}

The notation and terminologies in this paper are standard and the
reader is referred to \cite{Atlas, Guo2, LPS} if necessary. For a positive integer $n$, let $\pi(n)$ denote the set of prime divisors of $n$. Let $C_n$ denote a cyclic group of order $n$.  For a group $G$, denote by $|G|$ the order of $G$ and by $\pi(G)=\pi(|G|)$ the set of different primes dividing the order of $G$. We denote the socle of $G$ by $\Soc(G)$, which is the product of all minimal normal subgroups of $G$.  We say $G$ is almost simple if there exists a nonabelian simple group $T$ such that $T\leq G\leq \Aut(T)$. If $T$ is a simple Chevalley group, let $P_i$ denote the parabolic subgroup of $T$ obtained by deleting the $i$th node in the standard Dynkin diagram associated with $T$ (see \cite{LPS}). In
particular, other notation for simple groups mainly comes from
\cite{Atlas} and \cite{LPS}.

This paper is organized as follows. Section 2 shows that every almost simple group $X$ possesses a nonconjugate maximal subgroup pair $U$ and $V$ such that every maximal subgroup in this pair has nontrivial intersection with the socle of $X$. In Section 3, we prove a product-socle lifting result for a primitive permutation group with a unique nonsolvable minimal normal subgroup. In Section 4, we present the proof of main results.

\section{Preliminaries}

For a group $G$, a  \emph{maximal factorization} of $G$ is a
factorization $G=AB$ in which $A$ and $B$ are maximal subgroups of
$G$ (see \cite{LX}). If $T=\Soc(G)$, it is interesting to study
maximal factorizations of an almost simple group $G$ whose two factors do not contain $T$. In this case, the factors in the factorization are core-free and such a factorization is said to be nontrivial (see \cite{LX}). 
The maximal factorizations of almost simple groups have been
extensively investigated and there are lots of important results on
this topic (see \cite{LX, LPS}).

For a given group $G$, let $\cP(G)$ denote the following property:

\begin{center} \emph{whenever $M$ and $N$ are nonconjugate maximal subgroups of
$G$, one has $MN=G$.}
\end{center}
\begin{lemma}\label{lem:quotient}
For a group $X$, the property $\cP(X)$ is inherited by quotients of
$X$.
\end{lemma}

\begin{proof}
Let $K$ be a normal subgroup of $X$, and let $U/K$ and $V/K$ be
nonconjugate maximal subgroups of $X/K$. Then $U$ and $V$ are
maximal subgroups of $X$. If they were conjugate in $X$, their
images would be conjugate in $X/K$. Hence $U$ and $V$ are
nonconjugate in $X$  and $\cP(X)$ gives $UV=X$. Therefore
\begin{equation*}
(U/K)(V/K)=X/K.
\end{equation*}
Thus $\cP(X/K)$ holds.
\end{proof}

\begin{lemma}\label{lem:sylow-replacement}
Let $T<X\leq\Aut(T)$, where $T$ is nonabelian simple. Suppose that
$A<X$ is maximal and satisfies
\begin{equation*}\label{eq:Aassumptions}
X=TA,\qquad 1<A\cap T<T,
\end{equation*}
and suppose that $A$ is not a factor in any nontrivial maximal factorization of
$X$. Then there exists a maximal subgroup $B<X$ such that
\begin{equation*}
X=TB,\qquad 1<B\cap T<T,
\end{equation*}
$A$ and $B$ are nonconjugate, and $AB\neq X$.
\end{lemma}

\begin{proof}
Put $K=A\cap T$. Since $K<T$, there is a prime $p$ such that
\begin{equation*}
|K|_p<|T|_p.
\end{equation*}
Let $P$ be a Sylow $p$-subgroup of $T$. By the Frattini argument,
\begin{equation*}
X=T N_X(P).
\end{equation*}
By our assumptions, it is obvious that the subgroup $N_X(P)$ is
proper. Choose a maximal subgroup $B<X$ containing $N_X(P)$. Since
$X=T N_X(P)$, we have $X=TB$. Also
\begin{equation*}
P\leq B\cap T\leq T,
\end{equation*}
so $B\cap T$ contains a Sylow $p$-subgroup of $T$. In particular,
\begin{equation*}
|B\cap T|_p=|T|_p>|A\cap T|_p.
\end{equation*}
Thus $A$ and $B$ are not conjugate. Finally, $AB\neq X$, because $A$
is not a factor in  any nontrivial maximal factorization of $X$ by the
hypothesis.
\end{proof}

\begin{corollary}\label{cor:no-factorization}
Let $T<X\leq\Aut(T)$ with $T$ a nonabelian simple group. Suppose
that $X$ has no nontrivial maximal factorization. Then there exist nonconjugate
maximal subgroups $A,B$ of $T$ such that
\begin{equation*}
1<A\cap T<T,\qquad 1<B\cap T<T,
\qquad AB\neq X.
\end{equation*}
\end{corollary}

\begin{proof}
Choose a nontrivial Sylow subgroup $Q<T$. By the Frattini argument,
$X=T N_X(Q)$. Let $A$ be a maximal subgroup containing $N_X(Q)$.
Then $X=TA$ and
\begin{equation*}
1<Q\leq A\cap T<T.
\end{equation*}
Since $X$ has no nontrivial maximal factorization, the remaining assertions follow from
 Lemma~\ref{lem:sylow-replacement}.
\end{proof}

\begin{lemma}\label{lem:small-outer-auto}
Let $T<X\leq\Aut(T)$, where $T$ is nonabelian simple. Suppose that
$A<X$ is maximal such that $T$ is not contained in $A$. If $|X:T|<|A|$, then $A\cap T\neq 1$.
\end{lemma}

\begin{proof}
Since $T\nleq A$ by the hypothesis, we have $X=TA$. It follows that $$|X|=\frac{|T||A|}{|A\cap T|},$$ and so $|A|=|X:T||A\cap T|$.  By $|X:T|<|A|$, we see that $A\cap T$ is a nontrivial subgroup of $T$, as desired.
\end{proof}

The next proposition is the precise refinement of the almost simple case
needed later.

\begin{proposition} \label{prop:rigid-component}
Let $T$ be a nonabelian simple group and let
\begin{equation*}
T\leq X\leq\Aut(T).
\end{equation*}
There exist nonconjugate maximal subgroups $U,V<X$ such that, with
\begin{equation*}
K=U\cap T,\qquad L=V\cap T,
\end{equation*}
we have 
\begin{align}
X&=TU=TV, \label{eq:supplements}\\
X&\neq UV, \label{eq:proper-component-product}\\
1&<K<T,\qquad 1<L<T, \label{eq:nontrivial-intersections}\\
U&=N_X(K),\qquad V=N_X(L). \label{eq:normalizers}
\end{align}
Furthermore, $K$ is maximal among the proper $U$-invariant subgroups
of $T$, $L$ is maximal among the proper $V$-invariant subgroups of
$T$, and $K$ and $L$ are not $X$-conjugate.
\end{proposition}

\begin{proof}
By the main theorem in \cite{TT}, we have that there exists nonconjugate
maximal subgroups $U$ and $V$ such that $X\neq UV$ and $X=TU=TV$
since both $U$ and $V$ are not contained in $T$. Thus
\eqref{eq:supplements} and \eqref{eq:proper-component-product}
 are valid. Suppose first that both $K$ and $L$ are nontrivial proper subgroups of $T$.  Since $T\normal X$, we have $K\normal U$, and hence
\begin{equation*}
U\leq N_X(K).
\end{equation*}
  Since $U$ is maximal in
$X$, it follows that
\begin{equation*}
N_X(K)=U.
\end{equation*}
The same argument gives $N_X(L)=V$. Hence assertion \eqref{eq:normalizers} holds.

 From now on, we aim to prove assertion \eqref{eq:nontrivial-intersections}. If $X=T$, then claim
\eqref{eq:nontrivial-intersections}  is 
clear by setting $K=U$ and $L=V$.  
Next   we assume that $T<X$ and prove that $K$ and $L$ are nontrivial
subgroups, where $K=U\cap T$ and $L=V\cap T$. We proceed by case-by-case analysis by specifying the choices of nonconjugate maximal subgroups of $U$ and $V$. 

\begin{itemize}
\item Suppose that $T$ is a sporadic simple group. By \cite{Atlas}, we see that a sporadic simple group $T$ with nontrivial outer automorphism is as follows:
\begin{equation*}
M_{12}, M_{22}, J_2, J_3, Suz, HS, McL, He, Fi_{22}, Fi_{24}',   HN, O'N.
\end{equation*}
By \cite[Table 6]{LPS}, if $T=J_2, McL, Fi_{24}', HN, O'N, J_3$, then $X$ has no maximal factorization and 
Corollary~\ref{cor:no-factorization} applies so that the nonconjugate maximal subgroup pair $U$ and $V$ exists satisfying the condition of \eqref{eq:nontrivial-intersections}. If $T=M_{12}, M_{22}, Suz, HS, He, Fi_{22}$, then \cite[Table 6]{LPS} implies that $X$ has nontrivial maximal factorization, but \cite[Table 1]{TT} provides the nonconjugate maximal subgroups $U$ and $V$ of $X$ with $X\neq  UV$.   In this case, since $X=TU=TV$
for maximal subgroups $U$ and $V$ not containing $T$,  $|U|>|X:T|=2$ and $|V|>|X:T|=2$ by \cite{Atlas}, it follows from Lemma \ref{lem:small-outer-auto} that $U\cap T\neq 1\neq V\cap T$.

\item Suppose that $T=A_n$ with $n\geq 5$.  Let $X$ act on the set  $\Omega=\{1, 2, \ldots, n\}$.  If $n\neq 6$, then the point stabilizer $U$ of $\{1\}$ and
the two-set stabilizer $V$ of $\{1, 2\}$ are $S_{n-1}\cap X$ and $(S_{n-2}\times S_2)\cap X$ respectively.  By \cite[Theorem 2.4]{Wilson}, $U$ and $V$ are nonconjugate maximal subgroups of $X$. It is easy to see that $A_{n-2}\leq U\cap T$ and $A_{n-2}\leq V\cap T$. Therefore 
$U$ and $V$ have   nontrivial intersections with $T=A_n$. If $n=6$, then $A_6\simeq L_2(9)$ and this will be considered in the following cases.

\item Suppose $T=L_n(q)$.

 (I) Assume that $n=2$. If $q$ is even with $q\geq 8$ or $q$ is odd with $q>11$, then we choose different primes $s\in \pi(q-1)$ and $r\in
\pi(q+1)$ and let $S$ be a Sylow $s$-subgroup of $T$ and $R$ be a Sylow
$r$-subgroup of $T$. Set $U=N_X(S)$ and $V=N_X(R)$. Then by \cite[p.175]{TT}, we obtain that $U$ and $V$ are nonconjugate
maximal subgroups of $X$ such that $X=TU=TV$ and $X\neq UV$. By the
construction of $U$ and $V$, we see that $S\leq U\cap T\neq 1$ and $R\leq V\cap
T\neq 1$. 

Since $L_2(4)\simeq L_2(5)\simeq A_5$, we consider $q=7, 9, 11$ in the following.

If $T=L_2(7)$, then $X=L_2(7).2$. Take $U=D_{12}$ and $V=D_{16}$. Then $U$ and $V$ are nonconjugate maximal core-free subgroups of $X$ such that $X\neq UV$ (\cite[Table 2]{TT}).  Since $|X:T|=2$, $3\in \pi(U\cap T)$   and so $U\cap T\neq 1$. Also since $|V|>|X:T|=2$, $V\cap T\neq 1$. 

If $T=L_2(9)$, then the nonconjugate core-free maximal subgroup pair $U$ and $V$ such that $X\neq UV$ can be taken as follows (see \cite{Atlas}):

(i) when $X=L_2(9).2_1\simeq S_6$, $U=S_5$ and $V=S_4\times S_2$ as in the alternating group case;

(ii) when $X=L_2(9).2_2$, $U=D_{20}$ and $V=D_{16}$;

(iii) when $X=L_2(9).2_3$, $U=C_5:C_4$ and $V=C_8:C_2$;

(iv) when $X=L_2(9).2^2$, $U=C_{10}:C_4$ and $V$ is a Sylow $2$-subgroup of $X$.

In cases (ii)-(iv), since $|X:T|=2$ or $4$, $5\in \pi(U)\cap \pi(T)$, and $2^4||V|$, we get that $U\cap T\neq 1$ because an element of $U$ of order $5$ is contained in $T$,  and $V\cap T\neq 1$ by Lemma \ref{lem:small-outer-auto}.

If $T=L_2(11)$  and $X=L_2(11).2$, then let $U=S_4$ and $V=D_{24}$. Then $U$ and $V$ are nonconjugate maximal subgroups of $X$ such that $X\neq UV$ because $5\notin \pi(U)\cup \pi(V)$ (see \cite{Atlas}). Also, the Sylow $3$-subgroups of $U$ and $V$ are contained in $T$ as $|X:T|=2$.

 (II) Assume that $n=3$. Since $L_3(2)\simeq L_2(7)$,  we assume that $q\geq 3$. Choose a subgroup $H$ of $T$ with
\begin{equation*}
H=((C_{q-1}\times C_{q-1})/C_d)\mathbin{.}S_3, \qquad d=(3,q-1).
\end{equation*}
Let $U=N_X(H)$. Then,  by \cite[p.175]{TT},  $U$ is a maximal subgroup of $X$ and $U$ is absent from every nontrivial maximal factorization
of $X$. Thus $U\cap T\neq 1$. Now by 
Lemma~\ref{lem:sylow-replacement}, we can take the desired nonconjugate maximal subgroup $V$ of $X$ such that $V\cap T\neq 1$.

 (III) Assume that $n\geq 4$. 

 (i) Suppose first that $n$ is even. Since $L_4(2)\simeq A_8$ has been discussed, we let $q\geq 3$. Let $P_2$ be the parabolic subgroup of $L_n(q)$ obtained by deleting the second  node in the standard Dynkin diagram. Then by \cite[Table 1]{LPS}, $U=N_X(P_2)$ is a maximal subgroup of $X$ with $U\cap T=P_2\neq 1$ and is not a factor of every nontrivial maximal factorization of $X$.  It follows from
Lemma~\ref{lem:sylow-replacement} that the nonconjugate maximal
subgroup $V$ with required conditions exists. 

 (ii) Suppose next  that $n$ is odd. If $q=2$, then by \cite[Table 3]{LPS} and  \cite[p.70]{Atlas}, we can take $U=L_4(2).C_2$ and $V=C_{31}.C_{10}$ such that $U$ and $V$ are nonconjugate maximal subgroups of $X$ with $X\neq UV$ and $U\cap T\neq 1\neq V\cap T$. If $q\geq 3$, then by \cite[Table 1]{LPS}, either $X$ has no nontrivial maximal factorization or the maximal subgroup  $U=N_X(P_2)$ is not a factor in any nontrivial  maximal factorization of $X$. In both cases, we apply Lemma~\ref{lem:sylow-replacement} and  Corollary~\ref{cor:no-factorization}  to get the conclusion.

\item Suppose $T=\PSp_{2n}(q)$ with $n\geq 2$. 

 (I) Assume that $q$ is odd.  Assume first that  $T\neq \PSp_4(3)$ and $\PSp_6(3)$. By \cite[Table 1]{LPS}, one can take $U=N_X(\PSp_{2a}(q^b).b)$ and $V=N_X(P_2)$, where $n=ab$ and $b$ is a prime. Then $U$ and $V$ are nonconjugate core-free  maximal subgroups of $X$ and  $X\neq UV$. By the choice of $U$ and $V$, we see that $U\cap T\neq 1\neq V\cap T$.

 If $T= \PSp_4(3)$ and $X= \PSp_4(3).2$, then $X$ has the nonconjugate maximal subgroup pair $U$ and $V$ with the shape $U=3_+^{1+2}:2S_4$ and $V=3^3:(S_4\times 2)$ (\cite[Table 4]{TT} or \cite[p.26]{Atlas}). Since $5\notin \pi(U)$ and $5\notin \pi(V)$, we have $X\neq UV$. Since $|X:T|=2$ and $X$ has Sylow $3$-subgroups of order $3^4$, we have that $U\cap T\neq 1\neq V\cap T$.

 If $T=\PSp_6(3)$ and $X=\PSp_6(3).2$, then by \cite[p.113]{Atlas} one can take $U=2^{2+6}:3^3:D_{12}$ and $V=S_5$ as two nonconjugate maximal subgroups of $X$. Since $\{7, 13\}\subseteq \pi(X)$, we see that $X\neq UV$ by the orders of $U$ and $V$. Lemma \ref{lem:small-outer-auto} implies that $U\cap T$ and $V\cap T$ are nontrivial subgroups. 

 (II) Assume $q$ is even. Suppose first that $n=2$ and  $T=\PSp_4(q)$. Since $\PSp_4(2)\simeq S_6$, we let $q\geq 4$. If $q=4$, then $X=T.2$ or $T.4$. If $X=T.2$, let $U=(A_5\times A_5).2^2$ and let $V=S_6\times C_2$ (see \cite[Table 5]{TT}). Since $|X:T|\leq 4$ and the Sylow $5$-subgroup of $T$ has order $25$, we see that $U\cap T\neq 1$ and $V\cap T\neq 1$ because $5\in \pi(U)$ and $5\in \pi(V)$.  Similarly, we can discuss the case for $T.4$. Now suppose $q\geq 8$. By \cite[Table 2]{LPS}, if we take  $U=N_X(Sz(q))$ and
$V=N_X(O_4^-(q))$, then $X\neq UV$ and by \cite[Theorem 3.7]{Wilson}, $Sz(q)\leq T\cap U$ and  $O_4^-(q\leq T\cap V$.  Now Suppose that $n\geq 3$, then
$U=N_X(P_2)$ is not a factor in any maximal factorization of $X$ by \cite[Tables 1 and 2]{LPS} and therefore Lemma~\ref{lem:sylow-replacement} applies to this case.

\item Suppose  $T=U_n(q)$ with $n\geq 3$.

(I) Suppose that $n$ is odd. If $T\notin\{U_3(3), U_3(5), U_3(8), U_9(2)\}$, then $X$ has no nontrivial maximal factorization by \cite[Table 1]{LPS}. Here claim \eqref{eq:nontrivial-intersections} follows from Corollary \ref{cor:no-factorization}.

If $T\in\{U_3(3), U_3(5), U_3(8), U_9(2)\}$, then $2\leq |X:T|\leq 18$ and for any nonconjugate maximal subgroup pair $U$ and $V$ with $X\neq UV$, we have $|U|> 18$ and $|V|> 18$ by \cite{Atlas}. Hence, by Lemma \ref{lem:small-outer-auto}, we have $U\cap T$ and $V\cap T$ are nontrivial.

(II) Assume $n$ is even. If $T\neq U_4(2)$, $U_4(3)$, then $U=N_X(P_1)$ is not a factor in any nontrivial maximal factorization and $X=TU$, $U\cap T=P_1\neq 1$ and so we apply
Lemma~\ref{lem:sylow-replacement}. 

If $T=U_4(3)$, then $|X:T|\leq 8$ and any nonconjugate maximal subgroup pair $U$ and $V$ given in \cite[Table 7]{TT} has subgroups of order $3^4$. Lemma \ref{lem:small-outer-auto} implies  that  $U\cap T$ and $V\cap T$ are nontrivial.

Since $U_4(2)\simeq \PSp_4(3)$, the proof of this case is complete. 

\item Suppose $T=O_{2n+1}(q)$ with $n\geq 3$, $O_{2n}^+(q)$ with $n\geq 4$, $O_{2n}^-(q)$ with $n\geq 4$. 

In all these cases, by \cite[Tables 1-4]{LPS}, if we take $U=N_X(P_2)$, then $U$ is not a factor in any nontrivial maximal factorization with $X=TU$ and $U\cap T=P_2\neq 1$ and so $U$
satisfies the conditions of Lemma~\ref{lem:sylow-replacement}.

\item Suppose that $T$ is an exceptional Chevalley group of  Lie type.

 If  $T\in \{G_2(3^a), G_2(4), F_4(2^a)\}$, then choose $U=N_X(P_1)$ so that by \cite[Table 5]{LPS}, $T$ is not a factor in
 any nontrivial maximal factorization of $X$, and $X=UT$ and $U\cap T=P_1$. Therefore we can apply Lemma~\ref{lem:sylow-replacement}.
 The remaining exceptional groups have no maximal
 factorization and Corollary \ref{cor:no-factorization} applies.
\end{itemize}

We have therefore obtained nonconjugate maximal subgroups $U,V<X$
satisfying
\eqref{eq:nontrivial-intersections}. 

Put
$K=U\cap T$. Suppose that $K<J<T$ and that $J$ is a $U$-invariant subgroup. Then
$UJ$ is a subgroup properly containing $U$. Moreover,
\begin{equation*}
(UJ)\cap T=J(U\cap T)=J<T.
\end{equation*}
Thus $UJ<X$, contradicting the maximality of $U$. Hence $K$ is
maximal among the proper $U$-invariant subgroups of $T$. The proof
for $L=V\cap T$ is similar.

Finally, if $K$ and $L$ were conjugate in $X$, then their normalizers
$U=N_X(K)$ and $V=N_X(L)$ would be conjugate in $X$. This is
impossible. Hence $K$ and $L$ are not $X$-conjugate.
\end{proof}

\section{The product-socle lifting theorem}

We first recall some basic facts about finite primitive groups. Let
$G$ be a group acting faithfully and transitively  on a set
$\Omega$. $G$ is said to be primitive if $G$ does not have
nontrivial blocks. This is equivalent to  $G$ possessing  a core-free 
maximal subgroup $M$ such that $M_G=1$, where $M_G$ denotes the core
of $M$ in $G$, that is $M_G=\cap_{g\in G}M^g$ (see \cite[Definition 1.1.6]{BB} or \cite[Theorem 2.1.4]{Faw}). A nontrivial primitive
permutation group $G$ is isomorphic to a group that is either of affine
type, twisted wreath type, almost simple type, diagonal type, or
product type (see \cite[Theorem 2.7.1]{Faw}).

We now prove the key new structural step.

\begin{theorem}[Product-socle lifting]\label{thm:product-socle}
Let $G$ be a finite primitive group with a unique minimal normal
subgroup
\begin{equation}\label{eq:socle-product}
R=T_1\times\cdots\times T_k,
\end{equation}
where $k\geq 2$ and the $T_i$ are isomorphic nonabelian simple
groups. Then $G$ has nonconjugate maximal subgroups $A,B<G$ such
that
\begin{equation*}
AB\neq G.
\end{equation*}
\end{theorem}

\begin{proof}
We proceed with the proof via the following steps.

(I) \textit{Embed $G$ into the
permutational wreath product.}

It is clear that the factors $T_i$ form a single orbit under
conjugation by $G$ and $C_G(R)=1$. Fix $T_1\cong T$, and let $X$ be
the group induced by $N_G(T_1)$ on $T_1$. Since $C_G(R)=1$, the
inner automorphisms induced by $T_1$ identify $T$ with a normal
subgroup of $X$, and
\begin{equation*}
T\leq X\leq\Aut(T).
\end{equation*}
Let $P$ be the transitive permutation group induced by $G$ on the
set of components $\{T_1,\ldots,T_k\}$. After choosing
identifications $T_i\cong T$, one can embed $G$ into the
permutational wreath product
\begin{equation}\label{eq:wreath-embedding}
X\wr P=X^k\rtimes P,
\end{equation}
with $R=T^k$ as the inner base group; see, for example,
\cite[Section 4.3]{DM} or \cite[Section 2]{Faw}.

We use these
wreath-product coordinates throughout the rest of the proof.

(II) \textit{Construction of subgroups.}

Applying  Proposition~\ref{prop:rigid-component} to $T\leq
X\leq\Aut(T)$, we have that there are maximal subgroups $U,V<X$ and
subgroups
\begin{equation*}
K=U\cap T,\qquad L=V\cap T
\end{equation*}
with all the properties listed there, that is 

(i) $X=UT=VT$;

(ii) $X\neq UV$;

(iii) $K\neq 1$ and $L\neq 1$;

(iv) $U=N_X(K)$ and $V=N_X(L)$. 

 Define
\begin{equation*}
C=K^k,\qquad D=L^k,
\end{equation*}
and set
\begin{equation}\label{eq:AB-normalizers}
A=N_G(C),\qquad B=N_G(D).
\end{equation}

Since $U=N_X(K)$, we have $$N_{X\wr P}(C)=N_X(K)^k\rtimes P=U^k\rtimes P.$$
Intersecting with $G$, we have $$A=N_G(C)=G\cap (U^k\rtimes P).$$
The same argument holds for $B$. Consequently, we obtain 
\begin{equation}\label{eq:AB-wreath}
A=G\cap(U^k\rtimes P),\qquad
B=G\cap(V^k\rtimes P).
\end{equation}

(III) \textit{With the notation above, $G=RA=RB$.}

Take an arbitrary element $g\in G$. Let $$g=(x_1, \ldots, x_k)\pi$$ with every $x_i\in X$. By $X=TU$, for every $i$, we set $x_i=t_iu_i$ with $t_i\in T$ and $u_i\in U$. Write $$r=(t_1^{-1}, \ldots, t_k^{-1}).$$Then $r\in T^k=R$. Then $$rg=(u_1, \ldots, u_k)\pi.$$ In particular, $$rg\in U^k\rtimes P.$$ Since $rg\in G$, we have $rg\in A$ and so $$g=r^{-1}(rg)\in RA.$$ Thus $G\subseteq RA$. It follows that 
\begin{equation}\label{eq:GRA}
G=RA.
\end{equation}

Similarly, with $X=TV$, we have 
\begin{equation}\label{eq:GRB}
G=RB.
\end{equation}

Since $R$ acts trivially on the set of components, \eqref{eq:GRA}
and \eqref{eq:GRB} imply that both $A$ and $B$ induce the same
transitive permutation group $P$ as $G$.

(IV) \textit{With the same notation, $A\cap R=C$ and $B\cap R=D$.}

Since $A=N_G(C)$ and $R\leq G$, $A\cap R=N_R(C)$. Let $$r=(t_1, \ldots, t_k)\in T^k=R.$$ Then $$C^r=K^{t_1}\times \cdots \times K^{t_k}.$$
Thus $r$ normalizes $C$ if and only if every $t_i$ normalizes $K$. Therefore $$N_R(C)=N_T(K)^k.$$
Now $$N_T(K)=T\cap N_X(K)=T\cap U=K,$$ because $N_X(K)=U$ and $K=U\cap T$. Therefore $$A\cap R=N_R(C)=K^k=C.$$
The proof for $B\cap R=D$ is identical. Hence we have 
\begin{equation}\label{eq:intersections}
A\cap R=C,\qquad B\cap R=D.
\end{equation}

(V) \textit{The subgroups $A$ and $B$ are  maximal subgroups of $G$. }

Let
\begin{equation*}
A<J\leq G
\end{equation*}
and put
\begin{equation*}
E=J\cap R.
\end{equation*}
Then $C\leq E$, and $E$ is normalized by $A$.

Let $E_i$ denote the projection of $E$ on $T_i\cong T$.  In order to show that $E_i$ is $U$-invariant, we first identify the component action of $A$.  We claim that for each $i$, the group induced by $N_A(T_i)$ on $T_i$ is exactly $U$. One inclusion follows immediately from $A\leq U^k\rtimes P$. Conversely, let $u\in U$. Lift $u$ to an element of $N_G(T_i)$, say $g$. Since $G=RA$  by Step (III) and $R$ acts trivially on the set of $\{T_1, \ldots, T_k\}$, $T_i^g=T_i^a$ for some element $a\in A$. Thus every  $u\in U$
 actually occurs as the action of some element of $N_A(T_i)$, so   the group induced by $N_A(T_i)$ on the component $T_i$ is exactly   $U$. 

 Now $E$ is normalized by $A$ and therefore the projection $E_i$ is invariant under every element induced from $N_A(T_i)$. Since this induced group is $U$, we conclude that  $E_i$ is $U$-invariant in $T_i$. Since $C=K^k\leq E$, $E_i$
contains $K$ and so $$K\leq E_i\leq T_i.$$ 
By Proposition~\ref{prop:rigid-component}, we have that  $K$  is maximal among proper  
$U$-invariant subgroups of $T_i$, which implies that 
\begin{equation}\label{eq:projection-dichotomy}
E_i=K\quad\text{or}\quad E_i=T.
\end{equation}
The transitivity of $A$ on the components forces the same
alternative to hold for every $i$.

If $E_i=K$ for all $i$, then
\begin{equation*}
C\leq E\leq K^k=C,
\end{equation*}
so $E=C$. Since $E\normal J$, we obtain
\begin{equation*}
J\leq N_G(C)=A,
\end{equation*}
contrary to $A<J$. Thus this case gives $J=A$.

Suppose now that $E_i=T$ for every $i$. Put
\begin{equation*}
Q_i=E\cap T_i.
\end{equation*}
The inclusion $C=K^k\leq E$ gives
\begin{equation*}
1<K\leq Q_i.
\end{equation*}
For any $t\in T_i$, choose $e\in E$ whose $i$th projection is $t$,
which is possible because $E_i=T$. Conjugation by $e$ on $Q_i$
agrees with conjugation by $t$, since all other components
centralize $T_i$. As $E\normal J$, this shows that $Q_i\normal T_i$.
The simplicity of $T_i$ and the nontriviality of $Q_i$ imply
\begin{equation*}
Q_i=T_i.
\end{equation*}
Hence $E=R$. By \eqref{eq:GRA}, the subgroup $J$ contains
\begin{equation*}
RA=G,
\end{equation*}
so $J=G$. We have proved that $A$ is maximal. The same argument
proves that $B$ is maximal.

(VI) \textit{The subgroups $A$ and $B$ are nonconjugate in $G$.}

Suppose that $A^g=B$ for some $g\in G$. Since $R\normal G$,
equation \eqref{eq:intersections} gives
\begin{equation*}
C^g=(A\cap R)^g=B\cap R=D.
\end{equation*}
In the component representation, $g$ permutes the coordinates and
acts on each one by an element of $X$. Therefore the equality
$C^g=D$ implies that $K$ and $L$ are $X$-conjugate, contradicting
Proposition~\ref{prop:rigid-component}. Thus $A$ and $B$ are
nonconjugate.

(VII) \textit{The product $AB$ is proper in $G$. }

From
\eqref{eq:AB-wreath}, we know that an element $a\in A$ has the form$$a=(a_1, \ldots, a_k)\sigma$$ with every $a_i\in U$,  and an element $b\in B$ has the form $$b=(b_1, \ldots, b_k)\tau$$ with every $b_i\in V$. The multiplication rule in a wreath product means that the  top permutation of $ab$
is $\sigma\tau$, and the base coordinates of $ab$ are products of elements of $U$ and $V$.   Hence each
coordinate belongs to $UV$. Therefore 
\begin{equation}\label{eq:product-inclusion}
AB\subseteq (U^k\rtimes P)(V^k\rtimes P)
\subseteq (UV)^kP.
\end{equation}
This indicates that any element of $AB$ has every base coordinate lying in $UV$. Now choose $x\in X\setminus UV$. By the definition of $X$, there is an
element $g\in N_G(T_1)$ inducing $x$ on $T_1$. The top permutation
of $g$ fixes the first component, and the first base coordinate of
$g$ is $x$. Every element of the right-hand side of
\eqref{eq:product-inclusion} with the same top permutation has first
base coordinate in $UV$. Hence $g\notin AB$. Therefore
\begin{equation*}
AB\neq G.
\end{equation*}
This completes the proof.
\end{proof}

\section{Proof of the main theorems}

We now assemble the foregoing argument to prove our main result. For convenience, we restate Theorem ~\ref{thm:main} as follows. 

\begin{theorem}
A group $G$ with the property $\cP(G)$ is solvable.
\end{theorem}

\begin{proof}
Assume the assertion is false and let $G$ be a counterexample of
minimal order. By Lemma~\ref{lem:quotient}, every proper quotient of
$G$ is solvable.

If $R_1$ and $R_2$ were distinct minimal normal subgroups, then
\begin{equation*}
R_1\cap R_2=1,
\end{equation*}
and the homomorphism
\begin{equation*}
G\rightarrow G/R_1\times G/R_2
\end{equation*}
would be injective. Both quotient groups are solvable by minimality,
so their direct product is solvable, and hence $G$ would be
solvable. This is a contradiction. Let $R$ be the unique minimal
normal subgroup of $G$.

The quotient $G/R$ is solvable. The subgroup $R$ cannot be solvable,
since then $G$ would be an extension of solvable groups. A minimal
normal subgroup is characteristically simple, so
\begin{equation}\label{eq:R-main}
R=T_1\times\cdots\times T_k
\end{equation}
for isomorphic nonabelian simple groups $T_i$ and $C_G(R)=1$.

The Frattini subgroup $\Phi(G)$ is nilpotent, so it cannot contain
the nonsolvable subgroup $R$. Hence there is a maximal subgroup
$H<G$ with
\begin{equation*}
R\nleq H.
\end{equation*}
Since $R\normal G$ and $H$ is maximal,
\begin{equation*}
G=RH.
\end{equation*}
The core of $H$ in $G$ is trivial. For if $1\neq K\normal G$ and
$K\leq H$, then $K$ contains a minimal normal subgroup of $G$,
necessarily $R$, contradicting $R\nleq H$. Thus the action of $G$ on
the right cosets of $H$ is faithful and primitive. Its unique
minimal normal subgroup, and hence its socle $\Soc(G)$, is equal to  $R$.

If $k=1$, then $R$ is nonabelian simple and $G$  is an almost simple
group satisfying
\begin{equation*}
R\leq G\leq\Aut(R).
\end{equation*}
Proposition ~\ref{prop:rigid-component} supplies nonconjugate maximal subgroups $A,B<G$
with $AB\neq G$, contradicting $\cP(G)$.

If $k\geq 2$, by applying  Theorem~\ref{thm:product-socle} to the
faithful primitive action just constructed, we have that there are
nonconjugate maximal subgroups $A,B<G$ with $AB\neq G$,
contradicting $\cP(G)$.

Both cases are impossible. Therefore every finite group $G$ satisfying property 
$\cP(G)$ is solvable. 
\end{proof}

We now prove Theorem \ref{thm:main2}. In fact, we prove the following result.

\begin{theorem}
A group $G$ is solvable if and only if any two maximal subgroups of $G$ are $c$-permutable in $G$.
\end{theorem}

\begin{proof}
We first prove the necessity. Suppose that $G$ is a solvable group. Let $H$ and $K$ be any two maximal subgroups of $G$. If $H_G=K_G$, then $H$ and $K$ are conjugate in $G$ by \cite[Lemma 2.4.3]{Guo2}. Hence, for some $g\in G$, we have $H^g=K$ and so $H^gK=KH^g$, which implies that $H$ and $K$ are $c$-permutable in $G$. If $H_G\neq K_G$, then $H$ and $K$ are nonconjugate maximal subgroups of $G$. In this case, we either have $H_G$ is not contained in $K$ or have $K_G$ is not contained in $H$. Without loss of generality, we assume $H_G\nleq K$. Then $H_GK=HK=G$ as $K$ is a maximal subgroup. Hence $HK$ is permutable in $G$. Thus we have proved that $H$ and $K$ are $c$-permutable in $G$ in any case.

We next establish the sufficiency.  Suppose that any two maximal subgroups  of $G$ are $c$-permutable in $G$. We let $U$ and $V$ be any two nonconjugate maximal subgroups of $G$. Then there exist some element $x\in G$ such that $U^xV=VU^x$ is a subgroup of $G$. Since $U$ and $V$ are maximal and nonconjugate, we have that $U^xV=G$ and so $G=UV$ by \cite[Proposition 2.6]{LX}. This shows that $G$ has the property $\cP(G)$. Thus $G$ is solvable by Theorem \ref{thm:main}, as wanted.
\end{proof}

\section*{Acknowledgments}
Li was partially supported by the Natural Science Foundation of Suqian, China (K202432), the NSFC  (No. 12471017) and the Scientific Research Foundation for Advanced Talents of Suqian University (No. 2022XRC069). Yang was partially supported by a grant from the Simons Foundation (\#918096, to YY).

\section*{Disclosure Statement}
The authors declare that they have no competing interests and no conflicts of interest.

\section*{Data Availability Statement} Data sharing is not applicable to this article, as no data sets were generated or analysed during the current study.

\end{document}